\documentclass{amsart}
\usepackage{hyperref}
\usepackage{fullpage}
\usepackage{amsrefs}

\input xypic
\newdir{ >}{{}*!/-9pt/\dir{>}}

\newcommand{\sgn}	{\operatorname{sgn}}
\newcommand{\Ch}        {\operatorname{Ch}}
\renewcommand{\SS}      {\operatorname{SS}}

\newcommand{\CD}        {{\mathcal{D}}}
\newcommand{\CF}        {{\mathcal{F}}}
\newcommand{\CJ}        {{\mathcal{J}}}

\newcommand{\SpS}	{\operatorname{Sp}^\Sigma}
\newcommand{\Z}         {{\mathbb{Z}}}

\newcommand{\al}        {\alpha}
\newcommand{\bt}        {\beta} 
\newcommand{\gm}        {\gamma}
\newcommand{\dl}        {\delta}
\newcommand{\ep}        {\epsilon}
\newcommand{\tht}       {\theta}
\newcommand{\sg}        {\sigma}

\newcommand{\Sg}        {\Sigma}

\newcommand{\ot}        {\otimes}
\newcommand{\otb}	{\overline{\otimes}}

\newcommand{\tm}        {\times}
\newcommand{\xra}       {\xrightarrow}
\newcommand{\tge}	{\tau_{\geq 0}}
\newcommand{\tpsi}	{\widetilde{\psi}}

\renewcommand{\:}{\colon}

\newtheorem{theorem}{Theorem}

\newtheorem{lemma}[theorem]{Lemma}
\newtheorem{proposition}[theorem]{Proposition}

\theoremstyle{definition}
\newtheorem{remark}[theorem]{Remark}
\newtheorem{definition}[theorem]{Definition}

\begin{document}
\title{Is $D$ symmetric monoidal?}
\author{N.~P.~Strickland}

\maketitle 

Brooke Shipley has considered~\cite{sh:hza} a certain functor
$D\:\SpS(\Ch^+)\to\Ch$, and proved that it is strongly monoidal: we
have a natural isomorphism $\phi^{A,B}\:D(A)\ot D(B)\to D(A\ot B)$,
which is compatible with the associativity and unit isomorphisms in
$\SpS(\Ch^+)$ and $\Ch$.  The proof of one result in~\cite{sh:hza}
implicitly used the claim that $D$ is \emph{symmetric} monoidal, so we
have a commutative diagram
\[ \xymatrix{
     D(A) \ot D(B) \rto^{\phi^{A,B}} \dto_{\tau^{D(A),D(B)}} &
     D(A\ot B) \dto^{D(\tau^{A,B})} \\
     D(B)\ot D(A) \rto_{\phi^{B,A}} &
     D(B\ot A).
   }
\]
However, in a later note Brooke suggests that $D$ is not in fact
symmetric monoidal, so the relevant theorem needs to be weakened and
given a more elaborate proof.

I think that $D$ is in fact symmetric monoidal; this note is an
attempt to verify this.

We first give the relevant definitions.

\begin{definition}
 $\Ch$ is the category of chain complexes of abelian groups, with
 differential decreasing degrees by one.  $\Ch^+$ is the subcategory
 of chain complexes $C_*$ with $C_k=0$ for $k<0$.  These are
 symmetric monoidal categories, with the usual rules:
 \begin{align*}
  (A\ot B)_n &= \bigoplus_{n=i+j} A_i\ot B_j \\
  d(a\ot b) &= d(a)\ot b + (-1)^{|a|} a\ot d(b) \\
  \tau(a\ot b) &= (-1)^{|a||b|} b\ot a.
 \end{align*}
 We write $\Z[n]$ for the chain complex generated additively by a
 single element $e_n\in\Z[n]_n$ with $d(e_n)=0$.  For $n\geq 0$ we let
 $\Sg_n$ act on $\Z[n]$ and $\Z[-n]$ by $\sg_*(e_n)=\sgn(\sg)e_n$ and
 $\sg_*(e_{-n})=\sgn(\sg)e_{-n}$. 
\end{definition}

\begin{definition}
 We let $\CJ$ be the following symmetric monoidal category.  The
 objects are natural numbers, where $n$ is identified with
 $\{0,1,\dotsc,n-1\}$ in the usual way.  The morphisms are injective
 functions.  The monoidal product is given on objects by
 $n\Box m=n+m$.  For morphisms $f\:p\to q$ and $g\:m\to n$ we define 
 $f\Box g\:p+m\to q+n$ by 
 \[ (f\Box g)(i) =
      \begin{cases}
       f(i) & \text{ if } 0 \leq i < p \\
       g(i-p)+q & \text{ if } p\leq i < p+m.
      \end{cases}
 \]
 The associator maps $(m\Box n)\Box p\to m\Box(n\Box p)$ are
 identities, and the twist maps $\tau^{m,n}\:m\Box n\to n\Box m$ are
 given by 
 \[ \tau^{m,n}(i) =
     \begin{cases}
      i + n & \text{ if } 0\leq i < m \\
      i - m & \text{ if } m\leq i < m+n.
     \end{cases}
 \]
 This is a permutation of signature $(-1)^{nm}$.

 We write $\CF$ for the subcategory with the same objects as $\CJ$,
 where the morphisms are bijections.
\end{definition}

\begin{definition}
 We write $\SS(\Ch^+)$ for the category of functors $A\:\CF\to\Ch^+$.
 We make this a symmetric monoidal category by the usual procedure,
 due to Day: given $A,B\in\SS(\Ch^+)$ we have a functor 
 \[ \CF\tm\CF \xra{A\tm B} \Ch^+\tm\Ch^+ \xra{\ot} \Ch^+, \] and we
 take the left Kan extension along the functor $\Box\:\CF\tm\CF\to\CF$
 to get a functor $A\ot B\:\CF\to\Ch^+$.  This can be unwrapped as
 follows: Given $a\in A_{ni}$ and $b\in B_{mj}$ and a bijection
 $\tht\:n\Box m\to p$, we have an element
 $\tht_*(a\ot b)\in(A\ot B)_{p,i+j}$, and $A\ot B$ is generated by
 elements of this form, subject to the rule
 \[ (\tht\circ(\al\Box\bt))_*(a\ot b) = 
      \tht_*(\al_*(a)\ot\bt_*(b)).
 \]
 The twist map $\tau^{A,B}\:A\ot B\to B\ot A$ is given by 
 \[ \tau^{A,B}(\tht_*(a\ot b)) = 
     (\tht\circ\tau^{m,n})_*(\tau^{A_n,B_m}(a\ot b)) = 
     (-1)^{ij}(\tht\circ\tau^{m,n})_*(b\ot a). 
 \]
 The $\CF$-action is just
 $\gm_*(\tht_*(a\ot b))=(\gm\circ\tht)_*(a\ot b)$.  
\end{definition}

\begin{remark}
 If we tried to define the twist map by 
 \[ \tau^{A,B}(\tht_*(a\ot b)) = 
     \tht_*(\tau^{A_n,B_m}(a\ot b)) = 
     (-1)^{ij}\tht_*(b\ot a). 
 \]
 then this would not be well-defined (because of the relation 
 $(\tht\circ(\al\Box\bt))_*(a\ot b)=\tht_*(\al_*(a)\ot\bt_*(b))$) 
 and various other things would go wrong.  It is conceivable that this
 is behind the idea that $D$ is not symmetric monoidal.
\end{remark}

\begin{remark}
 If we have an isomorphism $\tht\:n\Box m\to p$ then we must have
 $p=n\Box m=n+m$, and in that context there is an obvious choice of
 $\tht$, namely the identity.  However, this is an artificial
 consequence of the fact that we are working with a skeletal category,
 and it is usually best to avoid thinking in those terms.  One
 exception, however, is that the above observation allows us to make
 sense of $\sgn(\tht)$, which is needed for our next definition.
\end{remark}

\begin{definition}
 We can regard $\Z[*]$ as an object of $\SS(\Ch^+)$ as follows.  All
 morphisms in $\CF$ are automorphisms and thus permutations; for
 $\al\in\CF(n,n)=\Sg_n$ we define $\al_*\:\Z[n]\to\Z[n]$ by
 $\al_*(e_n)=\sgn(\al)e_n$ as before.  We define
 $\mu\:\Z[*]\ot\Z[*]\to\Z[*]$ by
 $\mu(\tht_*(e_n\ot e_m))=\sgn(\tht)e_{n+m}$.
\end{definition}

\begin{lemma}
 The map $\mu$ makes $\Z[*]$ a commutative, associative and unital
 monoid in $\SS(\Ch^+)$.
\end{lemma}
\begin{proof}
 The only point that needs checking is commutativity.  We have
 \[ \tau(\tht_*(e_n\ot e_m))=
     (-1)^{nm}(\tht\tau^{m,n})_*(e_m\ot e_n), \]
 so
 \[ \mu\tau(\tht_*(e_n\ot e_m)) = 
     (-1)^{nm}\sgn(\tht\tau^{m,n})e_{n+m} = 
      \sgn(\tht)e_{n+m} = \mu(\tht_*(e_n\ot e_m)).
 \]
 Thus $\mu\tau=\mu$, as required.
\end{proof}

\begin{definition}
 An object $A\in\SpS(\Ch^+)$ is a module in $\SS(\Ch^+)$ over $\Z[*]$.  
 To unpack this a little, we have a sequence of chain complexes
 $A_n\in\Ch^+$ for $n\geq 0$, with an action of $\Sg_n$ on $A_n$
 and maps 
 \[ \sg\:\Z[k]\ot A_n \to A_{k+n} \]
 satisfying 
 \[ (\al\Box\bt)_*\sg(e_k\ot a) = \sgn(\al)\sg(e_k\ot \bt_*(a)) \]
 for all $\al\in\Sg_k$ and $\bt\in\Sg_m$, as well as the unit
 condition $\sg(e_0\ot a)=a$ and the associativity condition
 \[ \sg(e_i \ot \sg(e_j\ot a)) = \sg(e_{i+j}\ot a). \]
 Given $A,B\in\SpS(\Ch^+)$ we define $A\otb B$ to be their tensor
 product over $\Z[*]$, which is the coequaliser of a pair of maps
 $A\ot\Z[*]\ot B\to A\ot B$.  This makes $\SpS(\Ch^+)$ into a
 symmetric monoidal category.  The elements of $(A\ot B)_{nk}$ can
 again be written in the form $\tht_*(a\ot b)$, but there are some
 extra relations.  To be more explicit, consider a permutation
 $\phi\in\Sg_{n+m+p}$ (which we can regard as an isomorphism
 $(n+m)\Box p\to n+m+p$ or as an isomorphism $n\Box(m+p)\to n+m+p$)
 and elements $a\in A_{ni}$ and $b\in B_{pk}$.  We then have 
 \[ \phi_*(a\ot\sg(e_m\ot b)) =
     (-1)^{im}(\phi\circ(\tau^{m,n}\Box 1_p))_*(\sg(e_m\ot a)\ot b).
 \]
\end{definition}

\begin{definition}
 Given $A\in\SpS(\Ch^+)$, we define a functor $\CD(A)\:\CJ\to\Ch$ as
 follows.  On objects, we just put $\CD(A)_n=\Z[-n]\ot A_n$.  Now
 consider a morphism $\al\:n\to m$ in $\CJ$ (so $n\leq m$).  Let
 $\rho\:n\to m$ be given by $\rho(i)=(m-n)+i$.  One can choose
 $\al'\in\Sg_m$ with $\al'\rho=\al$.  We would like to define
 \[ \al_*(e_{-n}\ot a) = 
     \sgn(\al').(e_{-m}\ot\al'_*\sg(e_{m-n}\ot a)).
 \]
 To see that this is well-defined, let $\al''$ be another permutation
 such that $\al''\rho=\al$.  Then $\al'=\al''\circ(\dl\Box 1)$ for some
 $\dl\in\Sg_{m-n}$, so $\sgn(\al')=\sgn(\al'')\sgn(\dl)$ and
 \[ (\dl\Box 1)_*\sg(e_{m-n}\ot a) = 
     \sgn(\dl)\sg(e_{m-n}\ot a).
 \]
 It follows that
 \begin{align*}
   & \sgn(\al').(e_{-m}\ot\al'_*\sg(e_{m-n}\ot a)) \\
   &= \sgn(\al'')\sgn(\dl).(e_{-m}\ot\al''_*(\dl\Box 1)_*\sg(e_{m-n}\ot a)) \\
   &= \sgn(\al'').(e_{-m}\ot\al''_*\sg(e_{m-n}\ot a))
 \end{align*}
 as required.  One can check directly that this construction gives a
 chain map.

 We next check that we have a functor.  Suppose we have injective
 maps $n\xra{\al}m\xra{\bt}p$.  We can fit these in a diagram as
 follows: 
 \[ \xymatrix{
     n \rto^\al \drto_\rho &
     m \rto^\bt \drto_\rho & 
     p \\
     & m \uto^{\al'} \drto_\rho & 
     p \uto_{\bt'} \\
     & & p \uto_{1_{p-m}\Box \al'}
    }
 \]
 Here $\al'$ and $\bt'$ are bijections.  Put $\gm=\bt\al$ and
 $\gm'=\bt'\circ(1\Box\al')$.  These are related as required for the
 definition of $\gm_*$, so 
 \[ \gm_*(e_{-n}\ot a) 
     = \sgn(\gm').(e_{-p}\ot\gm'_*\sg(e_{p-n}\ot a)).
 \]
 Here $\sgn(\gm')=\sgn(\bt')\sgn(\al')$.  Also, we can write
 $\sg(e_{p-n}\ot a)$ as $\sg(e_{p-m}\ot\sg(e_{m-n}\ot a))$, and
 applying $(1\Box\al')_*$ to this gives
 $\sg(e_{p-m}\ot\al'_*\sg(e_{m-n}\ot a))$.  Thus
 \[ \gm_*(e_{-n}\ot a) =
     \sgn(\bt') e_{-p}\ot \bt'_*\sg(\sgn(\al')e_{p-m}\ot\al'_*\sg(e_{m-n}\ot a)).
 \]
 We see from the definitions that this is the same as
 $\bt_*\al_*(e_{-n}\ot a)$.

 Note that for $\bt\in\CJ(n,n)=\Sg_n$ we have
 $\bt_*(e_{-n}\ot a)=\sgn(\bt).e_{-n}\ot\bt_*(a)$, whereas
 $\rho_*(e_{-n}\ot a)=(-1)^{(m-n)n}e_{-m}\ot\sg(e_{m-n}\ot a)$.

 We now define $D(A)$ to be the colimit of the functor
 $\CD(A)\:\CJ\to\Ch$.  More explicitly, an element $a\in A_{n,k}$
 gives an element $e_{-n}\ot a\in\CD(A)_{n,k-n}$ and thus an element
 $\xi(a)\in D(A)_{k-n}$.  The complex $D(A)$ is generated by elements
 of this form, subject to the rules $\xi(\bt_*(a))=\sgn(\bt)\xi(a)$
 (for $\bt\in\Sg_n$) and $\xi(\sg(e_k\ot a))=\xi(a)$.  As
 $d(e_{-n}\ot a)=(-1)^n e_{-n}\ot da$, we have
 $d\xi(a)=(-1)^n\xi(da)$.  
\end{definition}

\begin{definition}
 We define $\phi^{A,A'}\:D(A)\ot D(A')\to D(A\otb A')$ as follows.  Given
 $a\in A_{n,i}$ and $a'\in A'_{n',i'}$, we have an element
 $\iota_*(a\ot a')\in(A\ot A')_{n+n',i+i'}$, and we define
 \[ \phi^{A,A'}(\xi(a)\ot\xi(a')) =
     (-1)^{n'n+n'i} \xi(\iota_*(a\ot a')). 
 \]
 To see that this is well-defined, we must check various identities.
 Firstly, given $\al\in\Sg_n$ and $\al'\in\Sg_{n'}$ we must have
 \[
  (-1)^{n'n+n'i} \xi(\iota_*(a\ot a')) 
    = (-1)^{n'n+n'i} \sgn(\al)\sgn(\al')\xi(\iota_*(\al_*(a)\ot\al'_*(a'))).
 \]
 Then, we must show that for $r,r'\geq 0$ we have
 \begin{align*}
  (-1)^{n'n+n'i} \xi(\iota_*(a\ot a')) &= 
   (-1)^{n'(n+r)+n'(i+r)} \xi(\iota_*(\sg(e_r\ot a)\ot a')) \\
  (-1)^{n'n+n'i} \xi(\iota_*(a\ot a')) &= 
   (-1)^{(n'+r')n+(n'+r')i} \xi(\iota_*(a\ot \sg(e_{r'}\ot a'))).
 \end{align*}
 For the first of our three identities, we note that
 \begin{align*}
  \iota_*(\al_*(a)\ot\al'_*(a')) &= 
   (\al\Box\al')_*\iota_*(a\ot a') \\
  \xi(\iota_*(\al_*(a)\ot\al'_*(a'))) &= 
   \xi((\al\Box\al')_*\iota_*(a\ot a')) =
   \sgn(\al\Box\al')\xi(\iota_*(a\ot a')) \\
   &= \sgn(\al)\sgn(\al')\xi(\iota(a\ot a')).
 \end{align*}
 For the second, we note that $(-1)^{n'(n+r)+n'(i+r)}=(-1)^{n'n+n'i}$ and
 that $\iota_*(\sg(e_r\ot a)\ot a')=\sg(e_r\ot\iota_*(a\ot a'))$, so
 $\xi(\iota_*(\sg(e_r\ot a)\ot a'))=\xi(\iota_*(a\ot a'))$.  The third
 identity can be deduced from the second using the identity
 \[ \iota_*(a\ot\sg(e_{r'}\ot a')) =
     (-1)^{ir'}(\tau^{r',n}\Box 1_{n'})_*(\sg(e_{r'}\ot a)\ot a')
 \]
 which holds in $(A\otb A')_{r'+n+n',r'+i+i'}$.  

 We next claim that $\phi^{A,B}$ is a chain map.  Indeed, we have 
 \begin{align*}
  d\phi^{A,A'}(\xi(a)\ot\xi(a'))
    &= (-1)^{n'n+n'i} d\xi(\iota_*(a\ot a')) \\
    &= (-1)^{n'n+n'i+n'+n}\xi(d(\iota_*(a\ot a')) \\
    &= (-1)^{n'n+n'i+n'+n}\xi(\iota_*(d(a)\ot a')) + 
       (-1)^{n'n+n'i+n'+n+i}\xi(\iota_*(a\ot d(a'))) \\
  \phi^{A,A'}(d(\xi(a)\ot\xi(a'))) 
    &= \phi^{A,A'}(d(\xi(a))\ot\xi(a')) +
       (-1)^{n+i}\phi^{A,A'}(\xi(a)\ot d(\xi(a'))) \\
    &= (-1)^n\phi^{A,A'}(\xi(d(a))\ot\xi(a')) + 
       (-1)^{n'+n+i}\phi^{A,A'}(\xi(a)\ot\xi(d(a'))) \\
    &= (-1)^{n+n'n+n'(i-1)}\xi(\iota_*(d(a)\ot a')) + 
       (-1)^{n'+n+i+n'n+n'i}\xi(\iota_*(a\ot d(a'))).
 \end{align*}
 These are easily seen to be the same.
\end{definition}

\begin{proposition}\label{prop-D-monoidal}
 $D$ is symmetric monoidal.
\end{proposition}
\begin{proof}
 We must show that the following diagram commutes:
 \[ \xymatrix{
      D(A) \ot D(B) \rto^{\phi^{A,B}} \dto_{\tau^{D(A),D(B)}} &
      D(A\otb B) \dto^{D(\tau^{A,B})} \\
      D(B)\ot D(A) \rto_{\phi^{B,A}} &
      D(B\otb A).
    }
 \]
 Consider $a\in A_{n,i}$ and $b\in B_{m,j}$ as before.  We have
 \[ \tau^{D(A),D(B)}(\xi(a)\ot\xi(b))=(-1)^{mn+mi+nj+ij}\xi(b)\ot\xi(a) \]
 and so 
 \begin{align*}
   \phi^{B,A}\tau^{D(A),D(B)}(\xi(a)\ot\xi(b)) 
    &= (-1)^{mn+mi+nj+ij}\phi^{A,B}(\xi(b)\ot\xi(a)) \\
    &= (-1)^{mn+mi+nj+ij}(-1)^{nm+nj}\xi(\iota_*(b\ot a)) \\
    &= (-1)^{mi+ij}\xi(\iota_*(b\ot a)).
 \end{align*}
 On the other hand, we have
 $\phi^{A,B}(\xi(a)\ot\xi(b))=(-1)^{mn+mi}\xi(\iota_*(a\ot b))$, and
 $\tau^{A,B}(\iota_*(a\ot b))=(-1)^{ij}\tau^{m,n}_*(b\ot a)$, so 
 \begin{align*}
  D(\tau)(\phi^{A,B}(\xi(a)\ot\xi(b))) 
   &= (-1)^{mn+mi}D(\tau^{A,B})(\xi(\iota_*(a\ot b))) \\
   &= (-1)^{mn+mi}\xi(\tau^{A,B}(\iota_*(a\ot b))) \\
   &= (-1)^{mn+mi+ij}\xi(\tau^{m,n}_*(b\ot a)) \\
   &= (-1)^{mn+mi+ij}\sgn(\tau^{m,n})\xi(\iota_*(b\ot a)) \\
   &= (-1)^{mi+ij}\xi(\iota_*(b\ot a)).
 \end{align*}
 It follows that the diagram commutes, as required.
\end{proof}

\begin{definition}
 Given $C\in\Ch$ we define $\tge C\in\Ch^+$ by 
 \[ (\tge C)_n = 
     \begin{cases}
      0 & \text{ if } n < 0 \\
      \ker(d\:C_0\to C_{-1}) & \text{ if } n = 0 \\
      C_n & \text{ if } n > 0.
     \end{cases}
 \]
 It is well-known that this gives a functor, right adjoint to the
 inclusion of $\Ch^+$ in $\Ch$.  
\end{definition}

\begin{definition}
 We define $R\:\Ch\to\SpS(\Ch^+)$ as follows.  Given $C\in\Ch$ and
 $n\geq 0$ we put $(RC)_n=\tge(\Z[n]\ot C)$.  We let $\Sg_n$ act on
 this via the action on $\Z[n]$, so
 $\al_*(e_n\ot c)=\sgn(\al).e_n\ot c$.  The map
 $\mu\:\Z[n]\ot\Z[m]\to\Z[n+m]$ induces a map
 $\sg\:\Z[n]\ot(RC)_m\to(RC)_{n+m}$, given by
 $\sg(e_n\ot(e_m\ot c))=e_{n+m}\ot c$.  This defines an object
 $RC\in\SpS(\Ch^+)$, and there is an obvious way to make it a
 functor. 
\end{definition}

\begin{lemma}
 There is a natural map $\psi=\psi^{C,C'}\:RC\otb RC'\to R(C\ot C')$
 given by 
 \[ \psi(\al_*((e_p\ot c)\ot(e_{p'}\ot c'))) = 
     (-1)^{p'k} \sgn(\al) e_{p+p'}\ot c\ot c'
 \]
 for all $c\in C_{p,k}$ and $c'\in C'_{p',k'}$.  Moreover, this makes
 $R$ a lax symmetric monoidal functor.
\end{lemma}
\begin{proof}
 First, we have a map 
 \[ \tpsi_{p,p'} = (-1)^{pp'}.
     (\Z[p]\ot C\ot \Z[p']\ot C'
      \xra{1\ot\tau\ot 1} 
      \Z[p]\ot\Z[p']\ot C\ot C'
      \xra{\mu\ot 1\ot 1}
      \Z[p+p']\ot C\ot C').
 \]
 If we restrict the domain to the subcomplex $(RC)_p\ot(RC')_{p'}$
 (which lies in $\Ch^+$) then the map must factor uniquely through
 $\tge(\Z[p+p']\ot C\ot C')=R(C\ot C')_{p+p'}$.  We thus get a map 
 \[ \tpsi_{p,p'} \: (RC)_p\ot(RC')_{p'} \to 
     R(C\ot C')_{p+p'}.
 \]
 This is easily seen to be equivariant for $\Sg_p\tm\Sg_{p'}$, or in
 other words, natural for $(p,p')\in\CF^2$.  Our description of
 $(RC)\ot(RC')$ as a Kan extension shows that $\tpsi$ induces a map
 $\psi\:(RC)\ot(RC')\to R(C\ot C')$, with the formula
 \[ \psi(\al_*((e_p\ot c)\ot(e_{p'}\ot c'))) = 
     (-1)^{pp'+p'(p+k)} \al_*(e_{p+p'}\ot c\ot c') = 
     (-1)^{p'k} \sgn(\al) e_{p+p'}\ot c\ot c'.
 \]
 We must show that this factors through $(RC)\otb(RC')$, or in other
 words, that the following diagram commutes:
 \[ \xymatrix{
     RC\ot\Z[*]\ot RC' \rrto^{\tau\ot 1} \dto_{1\ot\sg} & & 
     \Z[*]\ot RC\ot RC' \dto^{\sg\ot 1} \\
     RC\ot RC' \rto_\psi & 
     R(C\ot C') &
     RC\ot RC' \lto^\psi
    }
 \]
 The complex $(RC\ot\Z[*]\ot RC')_n$ is spanned by elements
 \[ x=\al_*((e_p\ot c)\ot e_m\ot (e_{p'}\ot c')) \]
 with $c\in C_k$ and $c'\in C'_{k'}$
 and $\al\:p\Box m\Box p'\to n$ in $\CF$.  For such $x$ we have 
 \[ (\tau\ot 1)(x) = 
     (-1)^{m(p+k)}(\al\circ(\tau^{m,p}\Box 1))_*
     (e_m\ot(e_p\ot c)\ot(e_{p'}\ot c')).
 \]
 Applying $\sg\ot 1$ (and recalling that
 $\tau^{m,p}_*e_{p+m}=(-1)^{mp}e_{p+m}$) gives 
 \[ (-1)^{mp+mk}(\al\circ(\tau^{m,p}\Box 1))_*
     ((e_{p+m}\ot c)\ot(e_{p'}\ot c')) = 
    (-1)^{mk}\al_*
     ((e_{p+m}\ot c)\ot(e_{p'}\ot c')).
 \]
 Applying $\psi$ now gives 
 \[ (-1)^{mk+p'k}\sgn(\al) e_{p+m+p'}\ot c\ot c'. \]
 Going the other way around the diagram, we first get 
 \[ (1\ot\sg)(x) = \al_*((e_p\ot c)\ot(e_{m+p'}\ot c')), \]
 and applying $\psi$ to this gives
 \[ (-1)^{(m+p')k}\sgn(\al) e_{p+m+p'}\ot c\ot c',  \]
 which is the same as before.  We thus have a natural map 
 \[ \psi=\psi^{C,C'}\:RC\otb RC'\to R(C\ot C') \]
 as claimed.  We leave it to the reader to check the relevant
 associativity property, showing that this makes $R$ a monoidal
 functor.  We will, however, check the commutativity property, which
 says that the following diagram commutes:
 \[ \xymatrix{
     RC\otb RC' \rto^\psi \dto_\tau &
     R(C\ot C') \dto^{R(\tau)} \\
     RC'\otb RC \rto_\psi &
     R(C'\ot C).
    }
 \]
 The group $(RC\otb RC')_n$ is spanned by elements
 $x=\al_*((e_p\ot c)\ot(e_{p'}\ot c'))$ with $c\in C_k$ and
 $c'\in C'_{k'}$ say.  For such $x$ we have
 \[ \tau(x) = (-1)^{(p+k)(p'+k')}(\al\circ\tau^{p',p})_*
     ((e_{p'}\ot c')\ot(e_p\ot c)).
 \]
 Applying $\psi$ to this gives $e_{p+p'}\ot c'\ot c$ multiplied by the
 sign
 \[ (-1)^{pp'+pk'+kp'+kk'}(-1)^{pk'}\sgn(\al)\sgn(\tau^{p',p}) = 
     (-1)^{kp'+kk'}\sgn(\al).
 \]
 Going the other way around, we have
 $\psi(x)=(-1)^{kp'}\sgn(\al)e_{p+p'}\ot c\ot c'$, and $R(\tau)$
 exchanges $c$ and $c'$ with a sign of $(-1)^{kk'}$, giving an overall
 sign of $kp'+kk'$ as before.
\end{proof}

\begin{lemma}
 There is a natural map $\eta\:A\to RD(A)$ given by
 $\eta(a)=e_n\ot\xi(a)$ for $a\in A_{n,k}$.  
\end{lemma}
\begin{proof}
 As $|\xi(a)|=k-n$ and $|e_n|=n$, the specified rule certainly gives a
 natural map of graded abelian groups.  We also have
 \[ d\eta(a) = (-1)^n e_n\ot d(\xi(a)) = (-1)^{2n} e_n\ot\xi(d(a)) =
     \eta(d(a)),
 \]
 so $\eta\:A_n\to(RD(A))_n$ is a chain map.  For $\al\in\Sg_n$ we have 
 \begin{align*}
  \al_*(\eta(a)) &= \al_*(e_n\ot\xi(a)) = \sgn(\al) e_n\ot\xi(a)\\
  \xi(\al_*(a))  &= \sgn(\al)\xi(a) \\ 
  \eta(\al_*(a)) &= e_n\ot\xi(\al_*(a)) = \sgn(\al) e_n\ot\xi(a);
 \end{align*}
 so $\eta$ is equivariant.  We also have
 $\sg(e_p\ot a)\in A_{p+n,p+k}$ with $\xi(\sg(e_p,a))=\xi(a)$, so 
 \begin{align*}
  \eta(\sg(e_p\ot a)) &= e_{p+n}\ot \xi(\sg(e_p\ot a)) \\
   &= e_{p+n}\ot \xi(a) = \sg(e_p\ot(e_n\ot\xi(a))) \\
   &= \sg(e_p\ot\eta(a)).
 \end{align*}
 Thus, $\eta$ is a morphism in $\SpS(\Ch^+)$.
\end{proof}

\begin{lemma}
 There is a natural map $\ep\:DR(C)\to C$ given by
 $\ep(\xi(e_n\ot c))=c$.
\end{lemma}
\begin{proof}
 We must first check that the definition is compatible with the
 relations $\xi(\al_*(a))=\sgn(\al)\xi(a)$ and
 $\xi(\sg(e_m\ot a))=\xi(a)$.  This is easy, because when $a=e_n\ot c$
 we have $\al_*(a)=\sgn(\al)a$ and $\sg(e_m\ot a)=e_{n+m}\ot c$.  We
 thus have a well-defined map of graded abelian groups, and 
 \[ \ep(d(\xi(e_n\ot c))) = 
     \ep((-1)^n\xi(d(e_n\ot c))) =
      (-1)^n\ep(\xi((-1)^n e_n\ot dc)) = 
       \ep(\xi(e_n\ot dc)) = dc = d\ep(\xi(e_n\ot c)),
 \]
 so $\ep$ is a chain map.  
\end{proof}

\begin{proposition}
 The following diagrams commute, showing that $D$ is left adjoint to
 $R$: 
 \[ \xymatrix{
    RC \rto^{\eta_{RC}} \drto_1 &
    RDRC \dto^{R\ep_C} & 
    DA \rto^{D\eta_A} \drto_1 & 
    DRDA \dto^{\ep_{DA}} \\
    & RC 
    & & DA
   }
 \]
\end{proposition}
\begin{proof}
 This is immediate from the formulae.
\end{proof}

\begin{definition}
 We let $\chi\:D(A\otb A')\to DA\ot DA'$ be conjugate to
 $\psi\:RC\otb RC'\to R(C\ot C')$.  More explicitly, $\chi$ is given
 by the composite
 \[ D(A\otb A') \xra{D(\eta\ot\eta)} 
    D(RDA\otb RDA') \xra{D(\psi)}
    DR(DA\otb DA') \xra{\ep}
    DA\ot DA'.
 \]
 From the definitions, we see that when $a\in A_{n,i}$ and
 $a'\in A_{n',i'}$ we have
 \[ \chi(\xi(\al_*(a\ot a'))) = 
     (-1)^{n'i+n'n}\sgn(\al) \xi(a)\ot\xi(a').
 \]
\end{definition}
\begin{remark}
 As $(R,\psi)$ is lax symmetric monoidal, it is formal to check that
 $(D,\chi)$ is lax symmetric comonoidal.
\end{remark}

\begin{lemma}
 $\chi\:D(A\otb A')\to DA\ot DA'$ is inverse to
 $\phi\:DA\ot DA'\to D(A\ot A')$.
\end{lemma}
\begin{proof}
 Recall the formula
 \[ \phi(\xi(a)\ot\xi(a')) = (-1)^{n'n+n'i} \xi(\iota_*(a\ot a')). \]
 It follows immediately that $\chi\phi=1$, and if we recall that
 $\xi(\al_*(a\ot a'))=\sgn(\al)\xi(\iota_*(a\ot a'))$, it also follows
 that $\phi\chi=1$.
\end{proof}

\begin{remark}
 As $(D,\chi)$ is symmetric comonoidal and $\chi$ is an isomorphism,
 it follows that $(D,\psi)=(D,\chi^{-1})$ is symmetric monoidal,
 giving an alternative proof of Proposition~\ref{prop-D-monoidal}.
\end{remark}

\end{document}